\documentclass[11pt]{amsart}

\usepackage[T1]{fontenc}
\usepackage{lmodern}
\usepackage{microtype}
\usepackage{amsmath,amssymb,amsthm,mathtools,mathrsfs}
\usepackage{enumitem}
\usepackage[hidelinks,pagebackref]{hyperref}

\renewcommand*{\backref}[1]{}
\renewcommand*{\backrefalt}[4]{%
\ifcase #1
Not cited%
\or
(Cited on page~#2)%
\else
(Cited on pages~#2)%
\fi
}
\usepackage[nameinlink,noabbrev]{cleveref}

\allowdisplaybreaks
\setlist[enumerate]{leftmargin=*,itemsep=0.22em,topsep=0.35em}
\setlist[itemize]{leftmargin=*,itemsep=0.18em,topsep=0.3em}

\newtheorem{theorem}{Theorem}[section]
\newtheorem{proposition}[theorem]{Proposition}
\newtheorem{lemma}[theorem]{Lemma}

\newtheorem{claim}[theorem]{Claim}
\theoremstyle{definition}
\newtheorem{definition}[theorem]{Definition}
\theoremstyle{remark}

\DeclareMathOperator{\SG}{SG}
\newcommand{\N}{\mathbb N}
\newcommand{\Nzero}{\mathbb N_0}
\newcommand{\Z}{\mathbb Z}

\newcommand{\Sd}{\mathcal S_d}
\newcommand{\cL}{\mathcal L}
\newcommand{\cW}{\mathcal W}
\newcommand{\supp}{\operatorname{supp}}
\newcommand{\one}{\mathbf 1}

\newcommand{\ulim}[2]{#1\!\operatorname{-lim}_{#2}}

\newcommand{\NilBohr}[1]{\mathrm{Nil}_{#1}\text{-}\mathrm{Bohr}_0}

\title[Nil-Bohr sets and sums with bounded gaps]
{Nil-Bohr Sets and Sums with Bounded Gaps}

\author[Y. C\MakeLowercase{ao}]{Y\MakeLowercase{ang} Cao}
\address[Y. Cao]{School of Mathematics and Statistics, Nanjing University of Science and Technology, Nanjing, 210094, Jiangsu, P.R. China}
\email{cy412@mail.ustc.edu.cn}

\author[J. Z\MakeLowercase{hao}]{J\MakeLowercase{ianjie} Zhao}
\address[J. Zhao]{School of Mathematics, Hangzhou Normal University, Hangzhou, 311121, Zhejiang, P.R. China}
\email{zjianjie@hznu.edu.cn}


\begin{document}

\begin{abstract}
We study the relation between $\NilBohr{d}$ sets and $\SG_d^*$ sets.  
We prove that every $\NilBohr{d}$ set is an $\SG_d^*$ set, answering a question of Host and Kra. More generally,  for a
compact Hausdorff distal system whose Ellis group admits a closed filtration
\[
E=E_1\supseteq E_2\supseteq\cdots\supseteq E_{d+1}=\{e\},
\  [E_j,E]\subseteq E_{j+1},
\]
we prove that every return time set is an $\SG_d^*$ set.
\end{abstract}
\maketitle

\section{Introduction}\label{sec:intro}

\subsection{The question of Host and Kra}\label{sec:intro-main}

Classical Bohr sets are described by return times for rotations on compact
abelian groups.  Host and Kra introduced $d$-step Nil-Bohr sets as a
higher order analogue, replacing the abelian setting by $d$-step nilsystems.
On the combinatorial side, they introduced the sets $\SG_d(P)$ of sums with
gaps, obtained from finite sums whose successive selected indices differ by at
most $d$ \cite{HostKra2011}.

The classical connections between recurrence and additive combinatorics go back
to Furstenberg and Weiss \cite{Furstenberg1981,FurstenbergWeiss1978}.
The finite sums side is rooted in Hindman's theorem \cite{Hindman1974}, while
piecewise Bohr structure was developed further by Bergelson, Furstenberg and
Weiss \cite{BergelsonFurstenbergWeiss2006}.  Higher order recurrence through
nilsystems and nilsequences was developed in
\cite{BergelsonHostKra2005,BergelsonLeibman2003,HostKra2005}.  Topological
recurrence in nilsystems and polynomial Nil-Bohr recurrence were studied in
\cite{HostKraMaass2016,Tu2014}.  Nil-Bohr sets were further characterized by
Huang, Shao and Ye in terms of generalized polynomials
\cite{HuangShaoYe2016}.

The two notions considered here arise from different settings.  A
$\NilBohr{d}$ set is defined in terms of return times in a $d$-step
nilsystem, while an $\SG_d^*$ set is defined by its intersections with
bounded-gap finite sum sets.  We recall the precise definitions below.

We fix some notation. Let $\N$ denote the positive integers and let $\N_0=\N\cup \{0\}$.
Let $d\geq1$ and let $P=(p_i)_{i\geq1}$ be an infinite
sequence of positive integers.  For a finite nonempty set
$\alpha\subseteq\N$, put
\[
 p_\alpha:=\sum_{i\in\alpha}p_i.
\]
Let $\Sd$ denote the family of all
$\alpha=\{i_1<\cdots<i_m\}\subseteq\N$ that satisfy, for $1\leq j<m$,
\[
 i_{j+1}-i_j\leq d.
\]
Set
\[
 \SG_d(P):=\{p_\alpha:\alpha\in\Sd\}.
\]
In the terminology of \cite{HostKra2011}, the zero blocks between two
selected indices have length strictly less than $d$.  A set $A\subseteq\Nzero$
is an $\SG_d^*$ set if, for every infinite sequence $P$ of positive integers,
\[
 A\cap\SG_d(P)\neq\varnothing
\]

A set $A\subseteq\Nzero$ is a $\NilBohr{d}$ set if there are a $d$-step
nilsystem $(X,T)$, a point $x\in X$, and a neighbourhood $U$ of $x$ such that
\[
 N(x,U)=\{n\in\Nzero:T^n x\in U\}\subseteq A.
\]

Host and Kra asked whether every $\NilBohr{d}$ set is an $\SG_d^*$ set
\cite[Question~2.11]{HostKra2011}.  Konieczny established the inclusion
\[
\NilBohr{d}\subseteq\SG_k^*
\]
for every $k\geq 4d$ \cite[Theorem~1.1]{Konieczny2017}.  The theorem below
establishes the inclusion with $k=d$, thereby answering the question of
Host and Kra in the affirmative.

\begin{theorem}\label{thm:nilbohr}
For every $d\geq1$, every $\NilBohr{d}$ set is an $\SG_d^*$ set.
Equivalently, if $A$ contains a neighbourhood return time set from a $d$-step nilsystem, then for every sequence $P=(p_i)_{i\geq1}$ of positive
integers there exists $\alpha\in\Sd$ with $p_\alpha\in A$.
\end{theorem}

Host and Kra proved that every $\SG_d^*$ set is strongly piecewise $\NilBohr{d}$ \cite[Theorem~2.10]{HostKra2011}. Together with Theorem~\ref{thm:nilbohr}, this yields
\[
 \NilBohr{d}
 \ \subseteq\ \SG_d^*
 \ \subseteq\ \text{strongly piecewise }\NilBohr{d}.
\]

\subsection{The dynamical theorem}\label{sec:intro-dynamical}

The combinatorial statement is a consequence of a more general recurrence principle for distal systems. By a \emph{topological dynamical system} or just a \emph{dynamical system}, we mean a pair $(X,T)$, where $X$ is a compact Hausdorff space and $T$ is a homeomorphism. We denote by $E(X, T)$ its Ellis semigroup. 
By a classical theorem of Ellis, $(X,T)$ is distal if and only if $E(X,T)$ is a group \cite{Ellis1969}. In this case, the Ellis semigroup $E(X,T)$ is a compact Hausdorff right topological group, which we call an \emph{Ellis group}.
 
Recall that a \emph{filtration} on an Ellis group $E$ is a descending sequence $\{E_j\}_{j \in \mathbb{N}}$ of closed subgroups of $E$ such that $E = E_1 \supseteq E_2 \supseteq \cdots$ and  $[E_j, E] \subseteq E_{j+1}$ for any $j \in \mathbb{N}$. 
 We assume the filtration terminates, i.e., $E_{d+1} = \{e\}$ for some $d$.
  The least such $d$ is called the \emph{length} of the filtration. We simply call $E$ a \emph{$d$-filtered Ellis group} if it is an Ellis group equipped with a filtration of length $d$.

Suppose that $(X,T)$ is distal and that its Ellis group $E$ admits a filtration $E=E_1 \supseteq E_2 \supseteq \cdots$, then the subgroup $E_2$ is normal and $E/E_2$ is abelian as $[E, E] \subseteq E_2$.  
Since $E_2$ is closed, $E/E_2$ is a compact Hausdorff group with the quotient topology. Its right translations are continuous, and, since $E/E_2$ is abelian, its left translations are continuous as well. Thus the multiplication is separately continuous, and hence $E/E_2$ is a compact topological group by \cite[Corollary~4.3]{QiuZhao2022}; see also \cite[Appendix~B, Corollary~B.18]{de Vries}.
We now state our main theorem as follows.

\begin{theorem}\label{thm:dynamical}
Let $(X,T)$ be a compact Hausdorff distal system whose Ellis group admits a
filtration of length $d$.  Then, for every $x\in X$, every
neighbourhood $U$ of $x$, and every sequence $P=(p_i)_{i\geq1}$ of positive
integers, there exists $\alpha\in\Sd$ such that
\[
 T^{p_\alpha}x\in U.
\]
Equivalently, every return time set $N(x,U)$ is an $\SG_d^*$ set.
\end{theorem}

Since finite products of nilsystems of step at most $d$ again have step at
most $d$, Theorem~\ref{thm:dynamical} also gives simultaneous recurrence in
finitely many such systems.

For nilsystems and systems of order $d$, the filtration comes from the topological commutator series of
the Ellis group.  Donoso proved the relevant topological nilpotence for
systems of order $d$ \cite{Donoso2014}, and Qiu and Zhao developed the
corresponding filtration and converse structure theory \cite{QiuZhao2022}. Thus Theorem~\ref{thm:dynamical} applies not only to nilsystems but also to other distal systems
with the same filtered Ellis-group structure.

\subsection{Strategy of the proof}\label{sec:intro-proof}

The proof centers on Theorem~\ref{thm:dynamical}.  We represent bounded gap sums by paths in a finite cube and use this representation to construct a compact subgroup of the Ellis group.  
We then use the filtration of the
Ellis group to obtain the recurrence asserted in
Theorem~\ref{thm:dynamical}.  Finally, we consider the orbit closure of a point in a $d$-step nilsystem
and use Theorem~\ref{thm:dynamical} to obtain
Theorem~\ref{thm:nilbohr}.

\subsection{Structure of the paper}\label{sec:intro-structure}

Section~\ref{sec:prelim} collects the facts about Ellis groups, topological
nilpotence, and ultrafilters used in the proofs.  Section~\ref{sec:cube}
develops the cube representation of bounded gap sums.
Section~\ref{sec:compact-subgroup} constructs the required compact subgroup of
the Ellis group.  Section~\ref{sec:proof} proves the two main theorems.  

\section{Ellis groups and ultrafilter compactifications}\label{sec:prelim}

\subsection{Nilpotent Groups, nilmanifolds and nilsystems}

Let $L$ be a group. For $g, h \in L$, the \emph{commutator} of $g$ and $h$ is defined by $ [g, h] = g h g^{-1} h^{-1}$.
For subsets $A, B \subseteq L$, we denote by $[A, B]$ the subgroup generated by $\{ [a, b] : a \in A,\ b \in B \}$. The commutator subgroups $ L_j, j \geq 1$ of $L$ is defined inductively by setting $ L_1 = L,  L_{j+1} = [L_j, L]$.
Let  $d \geq 1$ be an integer, we say that $L$ is \emph{$d$-step nilpotent} if $L_{d+1}$ is the trivial subgroup.

Let $L$ be a $d$-step nilpotent Lie group and let $\Gamma$ be a discrete cocompact subgroup of $L$. The compact manifold $X = L / \Gamma$ is called a \emph{$d$-step nilmanifold}. The group $L$ acts on $X$  by left translations and we write this action as $(g, x) \mapsto g x$. Fix $\tau \in L$ and let $T \colon X \to X$ be the transformation defined by $T(x) = \tau x$. Then $(X, T)$ is called a \emph{$d$-step nilsystem}.

A $d$-step nilsystem is a distal dynamical system. For such a system, the following properties are equivalent: transitivity and minimality. Moreover, the closure of any orbit in a $d$-step nilsystem, equipped with the restricted transformation, is itself a $d$-step nilsystem; hence it is minimal. See~\cite{AGH-1963} for proofs and further references.

A \emph{$d$-step pro-nilsystem} is defined as an inverse limit of rotations on $d$-step nilsystems. That is, $(Y, S)$ is a $d$-step pro-nilsystem if there exists an inverse system $\{(X_\alpha, T_\alpha)\}_{\alpha \in \mathcal{A}}$ of $d$-step nilsystems such that $(Y, S) \cong \varprojlim_{\alpha} (X_\alpha, T_\alpha)$.

\subsection{Ellis groups and topological nilpotence}\label{sec:ellis}

Let $(X,T)$ be a topological dynamical system.
Taking the closure in the topology of pointwise convergence, its Ellis semigroup is
\[
 E(X,T)=\overline{\{T^n:n\in\Z\}}\subseteq X^X.
\]
We use composition as multiplication, so $(pq)x=p(qx)$.  Every right translation
$p\mapsto pq$ is continuous.  The topological centre $\Lambda(E)$ consists
of those $p\in E$ for which left translation $q\mapsto pq$ is continuous.
For every $n\in\Z$, we have $T^n\in\Lambda(E)$.


Let $E$ be an Ellis group. For $A, B \subset E$, define $ [A, B]_{\mathrm{top}}$ as the closure of the subgroup $[A, B]$ spanned by $\{[a, b] : a \in A,\ b \in B\}$.
The topological commutator subgroups $E_j^{\mathrm{top}}$, $j \geq 1$, are defined by setting
$E_1^{\mathrm{top}} = E$ and $E_{j+1}^{\mathrm{top}} = [E_j^{\mathrm{top}}, E]_{\mathrm{top}}$.
It is easy to see that $E_j^{\mathrm{top}} \supset E_{j+1}^{\mathrm{top}}$ for all $j \geq 1$.
Let $d \in \mathbb{N}$. We say that $E$ is \emph{$d$-step top-nilpotent} if $E_{d+1}^{\mathrm{top}}$ is the trivial subgroup.

 The following result ensures that these recursively defined sets are indeed subgroups.
\begin{lemma}\cite{QiuZhao2022} \label{E_d^top subgroup}
  Let $E$ be an Ellis group. Then $E_d^{\mathrm{top}}$ is a subgroup of $E$ for all $d \in \mathbb{N}$.
\end{lemma}

Topological nilpotence is equivalent to the pro-nilsystem property; this equivalence is made precise by the following characterization.
 
\begin{lemma}\label{pro-nilsystem-top-nilpotent group} \cite{Donoso2014, QiuZhao2022}
Let $(X, T)$ be a minimal system, then it is a $d$-step pro-nilsystem if and only if its enveloping semigroup is a $d$-step top-nilpotent group.
\end{lemma}


The following lemma identifies when a compact subsemigroup of the Ellis group is a group.

\begin{lemma}\label{lem:compact-cancellative}
Let $K$ be a compact Hausdorff right-topological semigroup, Suppose that there exists a group $G$ such that $K\subseteq G$ 
and $xy=x\cdot_K y, x,y\in K$,
where $\cdot_K$ denotes the semigroup operation on $K$.
Then $K$ is a subgroup of $G$.  In particular, $K$ is a group.
\end{lemma}

\begin{proof}
Choose a minimal nonempty closed left ideal $I\subseteq K$.  Fix $a\in I$.
Right multiplication by $a$ is continuous, so $\varnothing\neq Ia\subseteq I$ and $Ia$ is closed.
Since $KIa\subset Ia$, and $I$ is minimal, one has $Ia=I$.
Hence there is $u\in I$ such that $ ua=a$.
Cancellation implies $u=e$.  Thus $e\in I$ and $K=Ke\subseteq I\subseteq K$,
and therefore $I=K$.

Now fix $a\in K=I$.  Then $Ka=Ia=I=K$. Thus there exists some $b\in K$ satisfies $ ba=e$.
Cancellation gives $b=a^{-1}\in K$, so $K$ is a group.
\end{proof}

\subsection{Ultrafilter and Uniform limits}\label{sec:ultrafilters}

We recall some standard facts concerning the ultrafilters of a discrete space; see, e.g., \cite[Chapters~3--4]{HindmanStrauss2012}.

Let $S$ be a discrete space.  We identify $\beta S$ with the set of all ultrafilters on $S$, and identify each $s\in S$ with the principal ultrafilter
\[
p_s:=\{A\subseteq S:s\in A\}.
\]
Thus, $S\subseteq\beta S$.  For $A\subseteq S$, set
\[
\overline{A}^{\,\beta S}:=\{p\in\beta S:A\in p\}.
\]
The family $\bigl\{\overline{A}^{\,\beta S}:A\subseteq S\bigr\}$
is a clopen base for the topology of $\beta S$.

Let $X$ be a topological space, let $p\in\beta S$, and let
$(x_s)_{s\in S}$ be a family of points of $X$.  For $x\in X$, $p\text{-}\lim_{s\in S}x_s=x$ 
if and only if for every neighborhood $U$ of $x$, $\{s\in S:x_s\in U\}\in p$.

 The following theorems summarize fundamental properties of  uniform limits via ultrafilters that will be used throughout the sequel without further reference.

\begin{theorem}\label{thm:ultralimit-basic} \cite{HindmanStrauss2012}
Let $S$ be a discrete space, let $p \in \beta S$, and let $\{x_s\}_{s \in S}$ be an indexed family in a topological space $X$.
\begin{enumerate}
  \item If $ p\text{-}\lim_{s\in S}x_s$ exists, then it is unique.
  \item If $X$ is compact, then $p\text{-}\lim_{s\in S}x_s $ exists.
\end{enumerate}
\end{theorem}

\begin{theorem}\label{thm:ultralimit-continuous}\cite{HindmanStrauss2012}
Let $S$ be a discrete space, let $p \in \beta S$, let $X$ and $Y$ be topological spaces, let $\{x_s\}_{s \in S}$ be an indexed family in $X$, and let $f \colon X \to Y$. If $f$ is continuous and $p\text{-}\lim_{s\in S}x_s $ exists, then
\[
   p\text{-}\lim_{s\in S} f(x_s) = f\bigl( p\text{-}\lim_{s\in S}x_s  \bigr).
\]
\end{theorem}

\begin{theorem}\cite{HindmanStrauss2012}
Let $S$ be a discrete space, let $X$ be a compact topological  space, let $f \colon S \to X$, and let $\tilde{f} \colon \beta S \to X$ be its continuous extension. Then for all $p \in \beta S$,
\[
  \tilde{f}(p) = p\text{-}\lim_{s\in S}  f(s).
\]
\end{theorem}

 We now introduce the notion of tensor product, following \cite{Luperi Baglini-2019}. For the case $S_1=S_2=S$, we refer to \cite{HindmanStrauss2012}.

 Let $S_1$ and $S_2$ be discrete spaces, and let $p\in\beta  S_1$ and $q\in\beta  S_2$. The \emph{tensor product} $p\otimes q\in\beta( S_1 \times  S_2)$ is defined by
\[
A\in p\otimes q \Longleftrightarrow\ 
\bigl\{x\in  S_1:\{y\in S_2:(x,y)\in A\}\in q\bigr\}\in p
\]
for every $A\subseteq  S_1 \times S_2$.
More generally, for $k\geq2$ and $p_i\in\beta S_i$ $(1\leq i\leq k)$,
the ultrafilter $p_1\otimes\cdots\otimes p_k \in
\beta\Bigl(\prod_{i=1}^k S_i\Bigr)$
is defined recursively by $p_1\otimes\cdots\otimes p_k :=(p_1\otimes\cdots\otimes p_{k-1})\otimes p_k$.

\begin{lemma}
Let $S_1, S_2$ be discrete spaces, let $p\in \beta S_1$ and $q\in \beta S_2$, and let $X$ be a compact Hausdorff topological space.
Then for every map $f:S_1\times S_2\to X$,
\[
(p\otimes q)\text{-}\lim_{(x,y)\in S_1\times S_2}f(x,y)
=p\text{-}\lim_{x\in S_1}q\text{-}\lim_{y\in S_2}f(x,y).
\]
\end{lemma}

\begin{proof}
 This follows immediately from the definition.
\end{proof}

 We recall the definition of the pushforward of an ultrafilter; for further details, we refer the reader to \cite{Benhamou2024, Lurie}.

Let $Z$ and $S$ be sets, and let $\phi:Z\to S$ be a map. For an ultrafilter $p\in\beta Z$, the \emph{pushforward} (or \emph{image}) of $p$ under $\phi$ is the ultrafilter
$\phi_*p$ on $S$ defined by
\[
\phi_*p:=\left\{A\subseteq S:\phi^{-1}(A)\in p\right\}.
\]
Equivalently,
\[
A\in\phi_*p\ \Longleftrightarrow\  \phi^{-1}(A)\in p.
\]

 The following standard property of pushforward ultrafilters will be used in in what follows.
\begin{lemma} \label{pushforward}
Let $Z$ and $S$ be sets, let $\phi:Z\to S$ be a map,
and let $p\in\beta Z$. If $X$ is a compact Hausdorff space,
then for every map $f:S\to X$,
\[
(\phi_*p)\text{-}\lim_{s\in S}f(s)
=p\text{-}\lim_{z\in Z}f(\phi(z)).
\]
\end{lemma}

\subsection{Adequate Partial Semigroups}
\label{sec:adequate-partial-semigroups}
We reall the definition of adequate partial semigroups, and some properties, see \cite{HindmanStrauss2012}.

Let $(S,\cdot)$ be a partial semigroup, that is, a set $S$
equipped with a partially defined associative operation. For
$s\in S$, put
\[
\varphi_S(s):=\{t\in S:s\cdot t\text{ is defined}\}.
\]
For a nonempty finite subset $F\subseteq S$, put
\[
\sigma_S(F):=\bigcap_{s\in F}\varphi_S(s).
\]
The partial semigroup $S$ is called \emph{adequate} if $\sigma_S(F)\neq\varnothing$ for every nonempty finite nonempty $F\subseteq S$.

For $s\in S$ and $A\subseteq S$, define
\[
s^{-1}A:=\{t\in\varphi_S(s):s\cdot t\in A\}.
\]

Identify $\beta S$ with the space of ultrafilters on the discrete
space $S$. Following \cite[Definition 4.22.1]{HindmanStrauss2012},
define
\[
\delta S:=\bigcap_{s\in S}
\overline{\varphi_S(s)}^{\,\beta S}.
\]
Equivalently,
\[
p\in\delta S\ \Longleftrightarrow\ \varphi_S(s)\in p, \text{ for every}s\in S.
\]
If $S$ is adequate, then $\delta S$ is nonempty and is a closed
subspace of $\beta S$.

For $p\in\beta S$ and $q\in\delta S$, define
\[
A\in p\cdot q \Longleftrightarrow 
\{s\in S:s^{-1}A\in q\}\in p, \ A\subseteq S.
\]
Equivalently,
\[
p\cdot q=p\text{-}\lim_{s\in S}
\left(q\text{-}\lim_{t\in\varphi_S(s)}s\cdot t\right).
\] 

\begin{theorem}\cite[Theorem 4.22.2]{HindmanStrauss2012}
  Let $(S, \cdot)$ be a partial semigroup. Then $\delta S \neq \varnothing$ if and only if $S$ is adequate, in which case $(\delta S, \cdot)$ is a compact semigroup where, for $p, q \in \delta S$ and $A \subseteq S$,
  \[
    A \in p \cdot q \  \text{if and only if} \  \{ s \in S : s^{-1}A \in q \} \in p.
  \]
\end{theorem}
 Moreover, if $(S,\cdot)$ is adequate, then $(\delta S,\cdot)$
is a compact right-topological semigroup.

\section{Sums with gaps and paths in the cube}\label{sec:cube}
In this section, we develop a path representation for the sets in $\Sd$. We first establish an identity that expresses $T^{p_\alpha}$ in terms of the elements $Q_t$. We then encode a finite set $\alpha$ by a state path in $E_d$, whose transitions determine the corresponding event sequence. Finally, we prove that every admissible timed path satisfying the non-return condition is uniquely realized by a set $\alpha\in\Sd$, and recover $T^{p_\alpha}$ from the events along the path.

\medskip

Fix $d\geq1$ and a sequence $P=(p_j)_{j\geq1}$ of positive integers.  We use
$\Sd$ and $p_\alpha$ as defined in Section~\ref{sec:intro}.  Empty sums below are zero.  For $t\geq1$, set
\[
  M_t=\sum_{\substack{1\leq j\leq t-d\\ j\equiv t\pmod d}}p_j,
 \  Q_t=T^{M_t}.  
\]
Let $\alpha\subseteq\N$ be finite.  
Let $a_j=0$ for $j\leq0$, $a_j=\one_\alpha(j)$ for $j>0$, where 
$\one_\alpha(j)$ is the indicator function of $\alpha$.
Then $a_j=0$ for all sufficiently large $j$ since $\alpha$ is finite. We call $(a_j)_{j\in\Z}$ the \emph{sequence associated with} $\alpha$.

We first record the following identity, which will be used later to express $T^{p_\alpha}$ in terms of the event sequence.

\begin{lemma}\label{lem:telescope}
For every finite $\alpha\subseteq\N$, let $(a_j)_{j\in\Z}$ be its associated sequence. Then
\[
 T^{p_\alpha}=\prod_{t\geq1} Q_t^{\,a_{t-d}-a_t}.   
\]
Only finitely many factors on the right are different from the identity.
\end{lemma}

\begin{proof}
Since $\alpha$ is finite, the product is in fact finite.
Choose $N$ so that $a_t=0$ for every $t>N$.  If $t>N+d$, then
\[
 a_{t-d}=a_t=0,
\]
so only finitely many factors are nontrivial.

Since $Q_t=T^{M_t}$, 
\[
 \prod_{t\geq1} Q_t^{\,a_{t-d}-a_t}
 =T^{\sum_{t\geq1}(a_{t-d}-a_t)M_t}.
\]
Expanding the exponent yields 
\[
\begin{aligned}
 \sum_{t\geq1}(a_{t-d}-a_t)M_t
 &=\sum_{t\geq1}(a_{t-d}-a_t)
   \sum_{\substack{1\leq j\leq t-d\\ j\equiv t\pmod d}}p_j\\
 &=\sum_{j\geq1}p_j
   \sum_{\substack{t= j+kd, \ k\ge 1 }}
   (a_{t-d}-a_t).
\end{aligned}
\]
For every $K\geq1$,
\[
 \sum_{k=1}^{K}
 \bigl(a_{j+(k-1)d}-a_{j+kd}\bigr)
 =a_j-a_{j+Kd}.
\]
Since $a_{j+Kd}=0$ for all sufficiently large $K$,
\[
 \sum_{\substack{t= j+kd, \ k\ge 1}}
 (a_{t-d}-a_t)=a_j.
\]
Therefore
\[
 \sum_{t\geq1}(a_{t-d}-a_t)M_t= \sum_{j\geq1}a_jp_j=\sum_{j\in\alpha}p_j =p_\alpha. 
\]
Hence
\[
 \prod_{t\geq1} Q_t^{\,a_{t-d}-a_t}=T^{p_\alpha}.
\]
\end{proof}

The exponent $a_{t-d}-a_t$ records precisely whether the state changes at time $t$. We now encode these changes by a path in the finite state space $E_d=\{0,1\}^d$. This will allow us to interpret the factors $Q_t^{a_{t-d}-a_t}$ as labels attached to the transitions of the path.

Set $E_d=\{0,1\}^d$ and index its coordinates by $\Z/d\Z$.
For a finite $\alpha\subseteq\N$, let $(a_j)_{j\in\Z}$ be its associated
sequence.  For $t\in\Nzero$ and $r\in\Z/d\Z$, let
\[
 j_r(t)=\max\{j\leq t:j\equiv r\pmod d\}.
\]
Define the map $x=x_\alpha:\Nzero\to E_d$ by
\[
 x(t)=\bigl(x_r(t)\bigr)_{r\in\Z/d\Z},
 \  x_r(t)=a_{j_r(t)}.
\]
We call $x(t)$ the \emph{state of the path associated with $\alpha$ at time $t$}. Thus, the state records the most recent value of the sequence $(a_j)$ in each residue class modulo $d$. In particular, a change of state can occur only when the corresponding residue class is updated.
Since $a_j=0$ for $j\leq0$, we have $x(0)=0_d$, where $0_d:=(0,\ldots,0)\in E_d$.

When passing from time $t-1$ to time $t$, only the coordinate
$r_t=t\pmod d$ can change, since
\[
 x_{r_t}(t-1)=a_{t-d}, \  x_{r_t}(t)=a_t.
\]
We denote a change $0\to1$ by $+r_t$ and a change $1\to0$ by $-r_t$. The signs are chosen so that the event labels reproduce exactly the corresponding factors in the telescoping identity. 
When $t\equiv r\pmod d$, define the corresponding labels by
\[
 v_t(+r)=Q_t^{-1}, \ v_t(-r)=Q_t.
\]
We denote by $\mathcal A_d=\{+r,-r:r\in\Z/d\Z\}$ the \emph{alphabet of events}.
 
For $r\in\Z/d\Z$, set
\[
 D_{+r}=\{x\in E_d:x_r=0\}, \  D_{-r}=\{x\in E_d:x_r=1\},
\]
and define the partial maps
\[
 \tau_{+r}:D_{+r}\longrightarrow E_d,
 \  \tau_{-r}:D_{-r}\longrightarrow E_d
\]
by
\[
 (\tau_{+r}(x))_s= \begin{cases}
  1,&s=r,\\
  x_s,&s\neq r,
 \end{cases}
 \ 
 (\tau_{-r}(x))_s=\begin{cases}
  0,&s=r,\\
  x_s,&s\neq r.
 \end{cases}
\]

We now record the times at which the state changes. Define
$$
C_\alpha=\{t\geq1:x(t-1)\neq x(t)\}
       =\{t\geq1:a_{t-d}\neq a_t\}.
$$
We call $C_\alpha$ the \emph{set of event times} of $\alpha$. Since $\alpha$ is finite and nonempty, $C_\alpha$ is finite and nonempty.
Write
\[
 t_-=\min C_\alpha, \  t_+=\max C_\alpha.
\]
Enumerate $C_\alpha$ as $\{t_1<\cdots<t_k\}$.  For $1\leq j\leq k$, let
$r_j\equiv t_j\pmod d$ and define
\[
 \varepsilon_j=
 \begin{cases}
  +r_j,&a_{t_j-d}=0,\ a_{t_j}=1,\\
  -r_j,&a_{t_j-d}=1,\ a_{t_j}=0.
 \end{cases}
\]
We call $(\varepsilon_1,\ldots,\varepsilon_k)$ the \emph{event symbols} of
$\alpha$.  The \emph{event sequence} of $\alpha$ is
\[
  \operatorname{Ev}(\alpha)
 =((t_1,\varepsilon_1),\ldots,(t_k,\varepsilon_k)).  
\]

The next lemma translates the defining condition for $\Sd$ into a condition on the associated state path. This characterization will be used to recognize when a timed path comes from a set in $\Sd$.

\begin{lemma}\label{lem:punctured}
Let $\alpha\subseteq\N$ be finite and nonempty.  Then $\alpha\in\Sd$ if and
only if $x(t)\neq0_d$ for every integer $t$ with $t_-<t<t_+$.
\end{lemma}

\begin{proof}
Let $\alpha=\{i_1<i_2<\cdots<i_m\}$, and let $(a_j)_{j\in \Z}$ be its associated sequence.
The coordinates of $x(t)$ are precisely
$a_{t-d+1},\ldots,a_t$, in some order.  Hence
\[
 x(t)=0_d  \Longleftrightarrow   \alpha\cap[t-d+1,t]= \varnothing.   
\]
Since $a_t = a_{t-d} = 0$ for all $t < i_1$ and
$(a_{i_1-d},a_{i_1})=(0,1)$, we have $t_-=i_1$.
Similarly, since $(a_{i_m},a_{i_m+d})=(1,0)$ and
$a_t=a_{t-d}=0$ for all $t>i_m+d$, it follows that $t_+=i_m+d$.

Assume first that $\alpha\in\Sd$.  Put $i_{m+1}=i_m+d$. 
We will show that for every integer $i_1<t<i_m+d$, $x(t)\neq 0_d$.
Otherwise, there exists $1\le j \le m$, such that $i_j<t<i_{j+1}$ and $x(t)=0_d$.
Then we have that
\[
 i_j\leq t-d<t<i_{j+1},
\]
and therefore $i_{j+1}-i_j>d.$
For $j<m$ this contradicts $\alpha\in\Sd$.  For $j=m$ it contradicts $i_{m+1}-i_m=d.$
Hence $x(t)\neq0_d$ whenever $t_-<t<t_+$.

Conversely, assume that $x(t)\neq0_d$ for every integer $t$ with
$t_-<t<t_+$. We will show that $\alpha \in \Sd$. Otherwise, there exists $1\le j<m$ such that $i_{j+1}-i_j>d$.
Let $t=i_j+d.$
Then
\[
 t_-\le i_j<t<i_{j+1}\leq i_m<t_+
\]
and
\[
 \alpha\cap[t-d+1,t]=\alpha\cap[i_j+1,i_j+d]=\varnothing.
\]
Thus, $x(t)=0_d$, a contradiction.  So for every $1\le j<m$ we have
\[
 i_{j+1}-i_j\leq d
\]
and hence $\alpha\in\Sd$.
\end{proof}

The preceding lemma shows that the condition $\alpha\in\Sd$ is equivalent to the associated state path avoiding $0_d$ between its first and last events. We now abstract this path structure and introduce admissible words and their timed realizations.

\begin{definition}
Let $e_0=(1,0,\ldots,0)\in E_d$.
A \emph{word} is a finite nonempty sequence
$w=(\varepsilon_1,\ldots,\varepsilon_k)$ for some $k\in \mathbb{N} $ with $\varepsilon_i\in\mathcal A_d$.  

The word $w=(\varepsilon_1,\ldots,\varepsilon_k)$ is \emph{admissible from $e_0$} if there are 
$x_0,x_1,\ldots,x_k\in E_d$ with $x_0=e_0$ and
$x_i=\tau_{\varepsilon_i}(x_{i-1})$ for every $1\leq i\leq k$.
 The \emph{state path} of $w$ is
\[
 e_0 \xrightarrow{\varepsilon_1 } x_1 \xrightarrow{\varepsilon_2 }x_2\xrightarrow{\varepsilon_3 } \cdots 
 \xrightarrow{\varepsilon_k } x_k,
\]
and simply $e_0 \xrightarrow{ w } x_k$.
We call $(x_0, x_1, \ldots, x_k)$ the \emph{state sequence} of $w$. The state sequence, when it exists, is unique.

Let $\cW$ be the set of words admissible from $e_0$ whose state sequence $(x_0, x_1, \ldots, x_k)$
satisfies $x_k=e_0$ and $x_i\neq0_d$ for every $0\leq i\leq k$.

Let $w=(\varepsilon_1,\ldots,\varepsilon_k)\in\cW$, where
$\varepsilon_j=\pm r_j$.  Choose integers $t_1<\cdots<t_k$ such that
$t_1\geq1$ and $t_j\equiv r_j\pmod d$ for every $j$.  The sequence
\[
 s=((t_1,\varepsilon_1),\ldots,(t_k,\varepsilon_k))
\]
is called a \emph{realization} of $w$.
Its \emph{support} is $\supp(s)=\{t_1,\ldots,t_k\}$.

Let $\cL$ denote the set of all realizations of words in $\cW$.  Suppose that
\[
 s=((t_1,\varepsilon_1),\ldots,(t_k,\varepsilon_k)),
 \ 
 u=((z_1,\eta_1),\ldots,(z_m,\eta_m))
\]
belong to $\cL$.  If $\max\supp(s)<\min\supp(u)$, define
\[
 su=((t_1,\varepsilon_1),\ldots,(t_k,\varepsilon_k),
     (z_1,\eta_1),\ldots,(z_m,\eta_m))
\]
and leave the product undefined otherwise.
These definitions provide the combinatorial class of timed paths that will be used below. In particular, $\mathcal L$ will serve as the underlying partial semigroup of realizations.

\end{definition}

Let
\[
 s=((t_1,\varepsilon_1),\ldots,(t_k,\varepsilon_k)),
 \ 1\leq t_1<\cdots<t_k,
\]
where $\varepsilon_j=\pm r_j, r_j\in \Z/ d\Z$ and $t_j\equiv r_j\pmod d$.  
Define
$$
x^s:\{0,1,\ldots,t_k\}\longrightarrow E_d
$$
recursively by setting $x^s(0)=0_d$ and, for $1\le t\le t_k$,
$$
x^s(t)=\begin{cases}
\tau_{\varepsilon_j}\bigl(x^s(t-1)\bigr),
& t=t_j\text{ for some }j\in\{1,\ldots,k\},\\[2mm]
x^s(t-1),
& t\notin\{t_1,\ldots,t_k\},
\end{cases}
$$
whenever the required transition is defined.
 We call $s$ an \emph{admissible timed path from $0_d$ to $0_d$} if 
\[
 x^s(t_j-1)\in\operatorname{Dom}(\tau_{\varepsilon_j}), \ 
  1\leq j\leq k 
\]
and $x^s(t_k)=0_d$.

 We can now prove the converse realization statement: every admissible timed path that avoids $0_d$ between its first and last events determines a unique set $\alpha\in\Sd$.

\begin{lemma}\label{lem:path-realisation}
Let $s=((t_1,\varepsilon_1),\ldots,(t_k,\varepsilon_k))$ be a nonempty admissible timed path from $0_d$ to $0_d$.  Suppose that $x^s(t)\neq0_d$ for
every integer $t$ with $t_1<t<t_k$.  Then there exists a unique
$\alpha\in\Sd$ such that
\[
  \operatorname{Ev}(\alpha)
 =((t_1,\varepsilon_1),\ldots,(t_k,\varepsilon_k)).  
\]
Moreover, 
\[
 T^{p_\alpha}=\prod_{j=1}^k v_{t_j}(\varepsilon_j).   
\]
\end{lemma}

\begin{proof}
Set $a_t=0$ for $t\leq0$.  For $t\geq1$, define
\[
 a_t=
 \begin{cases}
  1,&t=t_j\text{ and }\varepsilon_j=+r_j\text{ for some }j,\\
  0,&t=t_j\text{ and }\varepsilon_j=-r_j\text{ for some }j,\\
  a_{t-d},&t\notin\{t_1,\ldots,t_k\}.
 \end{cases}
\]

\begin{claim} \label{claim:a-s}
    For every $0\leq t\leq t_k$ and every $r\in\Z/d\Z$, 
\[
 a_{j_r(t)}=(x^s(t))_r.
\]
\end{claim}
 
\begin{proof}[Proof of the claim]
 We use induction on $t$.
For $t=0$, since $x^s(0)=0_d, a_j=0, j\leq0$, then for all $r\in\mathbb Z/d\mathbb Z$
\[
a_{j_r(0)}=0=(x^s(0))_r .
\]
 Assume that the statement holds for $t-1$. We prove it for $t$.
If $t=t_j$ for some $j\in\{1,\ldots,k\}$, then $x^s(t)=\tau_{\varepsilon_j}(x^s(t-1)).$
Since $t\equiv r_j\pmod d$, we have $j_{r_j}(t)=t$.
If $\varepsilon_j=+r_j$, then
\[
(x^s(t))_{r_j}=1=a_t=a_{j_{r_j}(t)}.
\]
If $\varepsilon_j=-r_j$, then
\[
(x^s(t))_{r_j}=0=a_t=a_{j_{r_j}(t)}.
\]
For $r\neq r_j$, $(x^s(t))_r=(x^s(t-1))_r$ and $j_r(t)=j_r(t-1)$,
Hence, by the induction hypothesis,
\[
a_{j_r(t)}=a_{j_r(t-1)}=(x^s(t-1))_r=(x^s(t))_r.
\]

Now suppose $t\notin\{t_1,\ldots,t_k\}$,
then $x^s(t)=x^s(t-1)$ and $a_t=a_{t-d}$.
If $t\not\equiv r\pmod d$, then $j_r(t)=j_r(t-1)$, and by the induction hypothesis, 
\[
a_{j_r(t)}=a_{j_r(t-1)}=(x^s(t-1))_r=(x^s(t))_r.
\]
If $t\equiv r\pmod d$, then $j_r(t)=t$ and $a_{j_r(t)}=a_t=a_{t-d}=a_{j_r(t-1)}$.
By the induction hypothesis, $a_{j_r(t-1)}=(x^s(t-1))_r$.
Thus $a_{j_r(t)}=a_{j_r(t-1)}=(x^s(t-1))_r=(x^s(t))_r$.
\end{proof}

Since $x^s(t_k)=0_d$, for every $r\in\mathbb Z/d\mathbb Z$,
by Claim \ref{claim:a-s}, we have $a_{j_r(t_k)}=(x^s(t_k))_r=0$.
Since $\{j_r(t_k):r\in\mathbb Z/d\mathbb Z\}=\{t_k-d+1,\ldots,t_k\}$,
it follows that
\[
a_{t_k-d+1}=\cdots=a_{t_k}=0.
\]
For $t>t_k$, $a_t=a_{t-d}$, thus for $1\leq m\leq d$, one has
\[
a_{t_k+m} =a_{t_k+m-d}=0.
\]
Repeating this argument inductively gives $a_t=0$ for all $t>t_k$.
 
Put $\alpha:=\{t\geq1:a_t =1\}$, then $\alpha $ is finite and nonempty since $x^s(t) \neq 0_d$ for $t_1<t <t_k$.
If $t\notin\{t_1,\ldots,t_k\}$, then $a_{t-d}=a_t$.
Assume that $t=t_j$ for some $1\le j \le k$, then 
\[
 a_{t_j}= (x^s(t_j))_{r_j} \neq  (x^s(t_j-1))_{r_j} =a_{t_j-d}.
\]
Therefore $C_\alpha=\{t_1,\ldots,t_k\}$. Set $t_-=t_1, t_+=t_k$.
 
The event symbol at $t_j$ is determined by
\[
 \varepsilon_j=
 \begin{cases}
 +r_j,&(a_{t_j-d},a_{t_j})=(0,1),\\
 -r_j,&(a_{t_j-d},a_{t_j})=(1,0).
 \end{cases}
\]
Thus $\operatorname{Ev}(\alpha)
 =((t_1,\varepsilon_1),\ldots,(t_k,\varepsilon_k))$.

For every $t$ with $0\leq t\leq t_k$, the state of the path $x^{\alpha}$ associated with $\alpha$  satisfies $x^\alpha(t)=x^s(t)$.
In particular, $x^\alpha(t)\neq0_d$ for every
integer $t$ with $t_-<t<t_+$.
Lemma~\ref{lem:punctured} implies $\alpha\in\Sd$.

For uniqueness, let $\beta\in\Sd$ satisfy $\operatorname{Ev}(\beta)=((t_1,\varepsilon_1),\ldots,(t_k,\varepsilon_k))$,
and let $(b_t)_{t\in \mathbb{Z}}$ be its  associated sequence.
Recall that $a_t=\mathbf 1_\alpha(t)$.
Since $\alpha$ and $\beta$ have the same event sequence, both
$c_t=a_t$ and $c_t=b_t$ satisfy
\[
c_t=\begin{cases}
1,&t=t_j\text{ and }\varepsilon_j=+r_j
   \text{ for some }j,\\
0,&t=t_j\text{ and }\varepsilon_j=-r_j
   \text{ for some }j,\\
c_{t-d},&t\notin\{t_1,\ldots,t_k\},
\end{cases}
\]
where $c_t=0$ for $t\leq0$.
We prove by induction on $t\geq1$ that $a_t=b_t$.
Suppose that $a_m=b_m, 1\leq m<t$.
If $t=t_j$ for some $j$, then,
\[
a_t=b_t=\begin{cases}
1,&\varepsilon_j=+r_j,\\
0,&\varepsilon_j=-r_j.
\end{cases}
\]
If $t\notin\{t_1,\ldots,t_k\}$, then $a_t=a_{t-d}=b_{t-d}=b_t$,
where the middle equality follows from initial condition if $t-d\le 0$ and from the induction hypothesis if $t-d\ge 1$.
Thus $a_t=b_t$ for every $t\geq1$, and consequently  
\[
\alpha=\{t\geq1:a_t=1\}=\{t\geq1:b_t=1\}=\beta.
\]
Hence $\alpha$ is unique.

Since $C_\alpha =\{t_1,\ldots,t_k\}$, $a_{t-d}-a_t=0$ for $t\notin C_\alpha$, and 
\[
 v_{t_j}(\varepsilon_j)
=Q_{t_j}^{\,a_{t_j-d}-a_{t_j}}, 1\le j\le k.  
\]
Hence,
\[
\begin{aligned}
\prod_{j=1}^k v_{t_j}(\varepsilon_j)
&=\prod_{j=1}^kQ_{t_j}^{\,a_{t_j-d}-a_{t_j}}\\
&=\prod_{t\in C_\alpha} Q_t^{\,a_{t-d}-a_t}\\
&=\prod_{t\ge1} 
Q_t^{\,a_{t-d}-a_t}\\
&=T^{p_\alpha},
\end{aligned}
\]
where the last equality follows from Lemma~\ref{lem:telescope}.

\end{proof}

\section{A compact subgroup of the Ellis group}\label{sec:compact-subgroup}
In this section, we construct a compact subgroup of the Ellis group that contains the
limits associated with all admissible words.  The construction proceeds
in three steps.  First, we verify that the ordered-support partial
semigroup $\cL$ is adequate and introduce its tail semigroup
$\delta\cL\subseteq\beta\cL$.  We then extend the map from $\cL$ into the
Ellis group to $\beta\cL$ and show that its restriction to $\delta\cL$
is a semigroup homomorphism whose image $K:=\Phi(\delta\cL)$
is a compact subgroup of the Ellis group.  Finally, we show that the
elements $h_{\varepsilon_1}\cdots h_{\varepsilon_k}$ associated with
admissible words all belong to $K$.

We map the ordered-support partial semigroup $\cL$ into the Ellis group.  Its
tail subsemigroup in $\beta\cL$ determines the compact subgroup used in the
proof of Theorem~\ref{thm:dynamical}.

We begin by recording the tail structure of $\cL$, which will be used to verify adequacy and to define $\delta\cL$.
Assume $d\geq2$.  For $N\in\N$ put
\[
 \cL_{>N}:=\{s\in\cL:\min\supp(s)>N\}.
\]
The product on $\cL$ is the ordered-support construction used for finite
objects in the theory of adequate partial semigroups.  
If
$F\subseteq\cL$ is finite and nonempty and $M_F:=\max_{s\in F}\max\supp(s)$,
then the common right multipliers of the elements of $F$ are
$\cL_{>M_F}$.  In this ordered-support setting, adequacy means that all these
tails are nonempty.

\begin{lemma}\label{lem:L-adequate}
Assume that $d\geq2$. For every $N\in \N$,  $\cL_{>N}$ is nonempty.  In particular, $\cL$ is an adequate partial semigroup.
\end{lemma}

\begin{proof}
Fix $N\in\N$.  Choose $r\in\Z/d\Z$ with $r\neq0$.  Then
\[
 e_0 \xrightarrow{+r} \tau_{+r}(e_0) \xrightarrow{-r}e_0,
\]
and $(\tau_{+r}(e_0))_0=1$. Hence $(+r,-r)\in\cW$.
Choose $t>N$ with $t\equiv r\pmod d$ and set
\[
 s=((t,+r),(t+d,-r)).
\]
Then $s\in\cL$ and $\min\supp(s)=t>N$,
so $s\in\cL_{>N}$.
Thus $\cL_{>N}\neq\varnothing$.

For a finite nonempty $F\subseteq\cL$, let
\[
 M_F=\max_{s\in F}\max\supp(s).
\]
Since $\mathcal L_{>M_F}\neq\varnothing$, there exists
$u\in\mathcal L$ such that $\min\supp(u)>M_F$.
Hence
\[
\max\supp(s)<\min\supp(u),\  s\in F,
\]
and consequently
\[
u\in\bigcap_{s\in F}\varphi_{\mathcal L}(s)\neq\varnothing.
\]
Thus $\mathcal L$ is adequate.
 
\end{proof}

With adequacy established, we now pass to the Stone--\v{C}ech
compactification of $\cL$ and isolate the ultrafilters that concentrate on arbitrarily far tails. We regard $\cL$ as a discrete space and let $\beta\cL$ denote its Stone--\v{C}ech compactification.
Let
\[
 \delta\cL:=\bigcap_{N\in\N}\overline{\cL_{>N}}^{\,\beta\cL}.
\]
An ultrafilter $q$ belongs to $\delta\cL$ if and only if
$\cL_{>N}\in q$ for every $N\in\N$.
The tails $\cL_{>N}$ form a decreasing family, and the preceding
adequacy argument shows that they provide arbitrarily far common right
multipliers for finite subsets of $\cL$.  Thus, by the discussion in
Section 2, $\delta\cL$ is a nonempty compact Hausdorff right-topological
semigroup.  Moreover, for $p\in\beta\cL$ and $q\in\delta\cL$,
\[
 pq=\ulim{p}{s\in\cL}\,
     \ulim{q}{t\in\cL_{>\max\supp(s)}}st.
\]

 We next construct the map from $\cL$ into the Ellis group that will produce the desired compact subgroup.

Let $s=((t_1,\varepsilon_1),\ldots,(t_k,\varepsilon_k))\in\cL$, where
$\varepsilon_j=\pm r_j$ with $r_j\in\mathbb Z/d\mathbb Z$.  Define
\[
 \Phi(s):=\prod_{j=1}^k v_{t_j}(\varepsilon_j)
=\prod_{j=1}^k\begin{cases}Q_{t_j}^{-1},&\varepsilon_j=+r_j,\\
Q_{t_j},&\varepsilon_j=-r_j, 
  \end{cases}.
\]
For every $1\leq j\leq k$, $v_{t_j}(\varepsilon_j)\in\Lambda(E)$, so $\Phi(s)\in\Lambda(E)$.  Whenever $su$ is defined,
$\Phi(su)=\Phi(s)\Phi(u)$. Since $\cL$ is discrete and $E$ is compact Hausdorff, $\Phi:\cL\to E$ extends uniquely to a continuous map $\Phi:\beta\cL\to E$.
 
The key point is that the tail semigroup is compatible with the
multiplicative structure of $\Phi$.
\begin{proposition}\label{prop:Phi-delta}
The restriction $\Phi:\delta\cL\longrightarrow E$
is a semigroup homomorphism.  Its image $K:=\Phi(\delta\cL) $
is a compact subgroup of $E$.
\end{proposition}

\begin{proof}
Let $p,q\in\delta\cL$.  For every $s\in\cL$, $\cL_{>\max\supp(s)}\in q$ and $\Phi(s)\in\Lambda(E)$.
Since $\Phi(s)\in\Lambda(E)$, left multiplication by $\Phi(s)$ is
continuous.  Since $E$ is right topological, right multiplication by $\Phi(q)$ is continuous as well.
 Hence
\[
\begin{aligned}
 \Phi(pq)
 &=\ulim{p}{s\in\cL}\ulim{q}{t\in\cL_{>\max\supp(s)}}\Phi(st)\\
 &=\ulim{p}{s\in\cL}\ulim{q}{t\in\cL}\Phi(s)\Phi(t)\\
 &=\ulim{p}{s\in\cL}\Phi(s)\Phi(q)\\
 &=\Phi(p)\Phi(q).
\end{aligned}
\]
Thus $\Phi|_{\delta\cL}$ is a semigroup homomorphism.

Since $\delta\cL$ is compact, so is $ K=\Phi(\delta\cL)$.
Moreover, $\Phi(p)\Phi(q)=\Phi(pq)\in K$ for all $p,q\in \delta\cL$, so $K$ is a subsemigroup of $E$.  
Since $E$ is right topological, so is $K$.
Hence $K$ is a compact Hausdorff right topological subsemigroup of the group $E$.
By Lemma~\ref{lem:compact-cancellative}, $K$ is a subgroup of $E$.

\end{proof}

We now show that $K$ contains the ultrafilter limits associated with admissible words.  For this purpose, we take independent free
ultrafilters along the congruence classes of the time parameters.

For $r\in\Z/d\Z$, put $D_r=\{t\geq1:t\equiv r\pmod d\}$
and fix a free ultrafilter $p_r\in\beta D_r\setminus D_r$.
Define
\[
L_r=\ulim{  p_r}{t\in D_r}Q_t^{-1},
 \ R_r=\ulim{  p_r}{t\in D_r}Q_t.    
\]
For $\varepsilon=\pm r$, set $h_{+r}=L_r,\ h_{-r}=R_r.$

\begin{lemma}\label{lem:word-limit}
Let $w=(\varepsilon_1,\ldots,\varepsilon_k)\in\cW$,  
then $h_{\varepsilon_1}\cdots h_{\varepsilon_k}\in K$.
\end{lemma}

\begin{proof}
 Write $\varepsilon_j=\pm r_j,\  r_j\in\mathbb Z/d\mathbb Z,
\ 1\leq j\leq k.$
For each $1 \le j\le k$, let $  p_{r_j} \in \beta D_{r_j} \setminus D_{r_j} $, where $D_{r_j} =\{t\geq1:t\equiv r_j\pmod d\}.$ 
Set $ \pi = p_{r_1}\otimes\cdots\otimes  p_{r_k}$ and let
\[
D=\left\{(t_1,\ldots,t_k)\in D_{r_1}\times\cdots\times D_{r_k}:t_1<\cdots<t_k\right\}.
\]
 For each $2\leq j\leq k$ and fixed $t_1,\ldots,t_{j-1}$, the set
\[
\{t_j\in D_{r_j}:t_j>\max\{t_1,\ldots,t_{j-1}\}\}
\]
is cofinite in $D_{r_j}$ and hence belongs to $p_{r_j}$.  It follows,
by the definition of the tensor product, that $D\in p$.
Hence $ \pi|_D:=\{A\cap D:A\in  \pi\}$
is an ultrafilter on $D$, so $ \pi|_D\in\beta D$.

For $(t_1,\ldots,t_k)\in D$, put $s(t_1,\ldots,t_k)=((t_1,\varepsilon_1),\ldots(t_k,\varepsilon_k))$.
Since $w\in\mathcal W$, this is a realization of $w$, and hence $s(t_1,\ldots,t_k)\in\mathcal L$. 
Thus $s:D\longrightarrow\mathcal L$ is well-defined.
Let $ \mathfrak p=s_*(  \pi|_D)$
be the pushforward ultrafilter on $\cL$.
 
 We next verify that $\mathfrak p\in \delta\cL$.
Fix $N\in\N$.  Since $ p_{r_1}$ is free, $D\cap\{(t_1,\ldots,t_k):t_1>N\}\in p$.
Also,
\[
 D\cap \{(t_1,\ldots,t_k):t_1>N\} \subseteq s^{-1}(\cL_{>N}).
\]
Hence $\cL_{>N}\in  \mathfrak p$.
Since this holds for every $N$, $ \mathfrak p \in\delta\cL$.

By Theorem \ref{pushforward}, one has that
\[
\begin{aligned}
 \Phi( \mathfrak p) &= \ulim{  \mathfrak p }{ u\in\mathcal L} \Phi(u)=\ulim{ \pi}{ (t_1,\ldots,t_k)\in D}\Phi(s(t_1, \ldots, t_k))\\
 &= \ulim{ \pi }{ (t_1,\ldots,t_k)\in D}\prod_{j=1}^k
v_{t_j}(\varepsilon_j)\\
 &=\ulim{ p_{r_1}}{t_1\in D_{r_1}}\cdots
   \ulim{  p_{r_k}}{t_k\in D_{r_k}}
   \prod_{j=1}^k v_{t_j}(\varepsilon_j).
\end{aligned}
\]
 
Set $H_{k+1}=e$.  For $1\leq j\leq k$, define $H_j=h_{\varepsilon_j}H_{j+1}$.
For fixed $t_1,\ldots,t_{j-1}$, put 
\[
 A_{j-1}=\prod_{i=1}^{j-1}v_{t_i}(\varepsilon_i), \ A_0:=e.   
\]
Then $A_{j-1}\in\Lambda(E)$.  Right multiplication by $H_{j+1}$ is continuous.  Therefore
\[
\begin{aligned}
 &\ulim{ p_{r_j}}{t_j\in D_{r_j}}
 A_{j-1}v_{t_j}(\varepsilon_j)H_{j+1}\\
  =&A_{j-1}\left(\ulim{ p_{r_j}}{t_j\in D_{r_j}}v_{t_j}(\varepsilon_j)\right)H_{j+1}\\
  =&A_{j-1}h_{\varepsilon_j}H_{j+1} = A_{j-1}H_j.
\end{aligned}
\] 
Applying this identity successively for $j=k,k-1,\ldots,1$ gives 
\[
 \Phi( \mathfrak p)=H_1=h_{\varepsilon_1}\cdots h_{\varepsilon_k}. 
\]
Since $ \mathfrak p\in\delta\mathcal L$, it follows that
$h_{\varepsilon_1}\cdots h_{\varepsilon_k}=\Phi( \mathfrak p)\in K$.
 
\end{proof}

\section{Nilpotent cancellation and recurrence}\label{sec:proof}

 The purpose of this section is to use the nilpotent structure of the Ellis group to produce the identity as a limit of the transformations associated with admissible finite patterns.  
  We first establish the following consequence
of the filtration condition.  We then apply this result to the compact subgroup $K$ constructed in Section~\ref{sec:compact-subgroup}, show that the identity belongs to the closure of the relevant return transformations,
and deduce the desired recurrence statement.  Finally, we apply
Theorem~\ref{thm:dynamical} to nilsystems to obtain the desired nil-Bohr conclusion.

 We begin with the following consequence of the filtration condition.

\begin{lemma}\label{lem:cap}
Let $G$ be a group, and let $G_1,\ldots,G_{d+1}$ be subgroups satisfying $[G_j,G_1]\subseteq G_{j+1}$ for every $1\leq j\leq d$.  Assume that these subgroups form the filtration
\[
 G=G_1\supseteq G_2\supseteq\cdots\supseteq G_{d+1}=\{e\}.
\]
Let $C\in G_2$ and $\ell_1,\ldots,\ell_{d-1}\in G_1$.  For
$U=\{i_1<\cdots<i_k\}\subseteq\{1,\ldots,d-1\}$ set
\[
 B_U=\ell_{i_k}\cdots\ell_{i_1},
 \ z_U=B_UCB_U^{-1},
\]
and put $z_\varnothing=C$.  Then
\[
  C\in \big\langle z_U:\varnothing\neq U\subseteq\{1,\ldots,d-1\}\big\rangle.  
\]
 
\end{lemma}

\begin{proof}
Set $W_0=C$ and, for $1\leq m\leq d-1$, define
\[
 W_m=[\ell_m,W_{m-1}]=\ell_mW_{m-1}\ell_m^{-1}W_{m-1}^{-1}.   
\]
For each $1\leq m\leq d-1$, if $W_{m-1}\in G_{m+1}$, then 
\[
 W_m =[\ell_m,W_{m-1}]=[W_{m-1},\ell_m]^{-1} \in G_{m+2}.
\]
Since $W_0=C\in G_2$, induction shows that for every $m$ with
$0\leq m\leq d-1$, $ W_m\in G_{m+2}$.
In particular, $W_{d-1}=e$.


We use the convention $\{1,\ldots,0\}:=\varnothing$.
For every $0\le m\le d-1$, we consider the recursive formal expansion of $W_m$ obtained from
\[
W_m=\ell_mW_{m-1}\ell_m^{-1}W_{m-1}^{-1}.
\]
For $U\subseteq\{1,\ldots,m\}$, let $\epsilon_m(U)$ denote the exponent of the corresponding formal factor $z_U$ in this expansion. We verify inductively that every $z_U$ occurs exactly once and $\epsilon_m(U)=(-1)^{m-|U|}$.
For $m=0$, $W_0=z_\varnothing, \epsilon_0(\varnothing)=1$.
Assume the assertion for $m-1$. We show the case of $m$.
For every $U\subseteq\{1,\ldots,m-1\}$,  For every $U\subseteq\{1,\ldots,m-1\}$,
\[
\ell_mz_U\ell_m^{-1}=\ell_mB_UCB_U^{-1}\ell_m^{-1}
=B_{U\cup\{m\}}CB_{U\cup\{m\}}^{-1}=z_{U\cup\{m\}}.
\]
Thus  $\ell_mW_{m-1}\ell_m^{-1}$ contains exactly once each $z_V^{\epsilon_m(V)}$ with $m\in V$,
while $W_{m-1}^{-1}$ contains exactly once each $z_V^{\epsilon_m(V)} $ with $m\notin V$.
Where  
\[
 \epsilon_m(V)=\begin{cases}
  \epsilon_{m-1}(V\setminus\{m\}),&m\in V,\\
  -\epsilon_{m-1}(V),&m\notin V.
 \end{cases}
\]
In both cases, $\epsilon_m(V)=(-1)^{m-|V|}$. 

In particular, for $m=d-1$, the unique formal factor associated with $U=\varnothing$ is
\[
z_{\varnothing}^{(-1)^{d-1}}=C^{(-1)^{d-1}},
\]
while every other formal factor belongs to
\[
H:=\big\langle z_U:\varnothing\neq U\subseteq\{1,\ldots,d-1\}
\big\rangle.
\]
Thus there exist $A,B\in H$ such that
\[
W_{d-1}=A\,C^{(-1)^{d-1}}B.
\]
Since $W_{d-1}=e$, we have $C^{(-1)^{d-1}}=A^{-1}B^{-1}\in H$,
and consequently $C\in H$. 
  
\end{proof}

We are now ready to prove our mian theorems.

\begin{proof}[Proof of Theorem~\ref{thm:dynamical}]
Assume first that $d\geq2$.  Fix $P=(p_i)_{i\geq1}$, and use the
notation from Sections~\ref{sec:cube} and~\ref{sec:compact-subgroup}.
Let $\pi:E\longrightarrow E/E_2$ be the quotient map.  Inversion is continuous
on the compact topological
group $E/E_2$.  Hence, for every $r\in\Z/d\Z$,
\[
\begin{aligned}
 \pi(L_r)
 &=\ulim{  p_r}{t\in D_r}\pi(Q_t^{-1})\\
 &=\ulim{ p_r}{t\in D_r}\pi(Q_t)^{-1}\\
 &=\left(\ulim{\ p_r}{t\in D_r}\pi(Q_t)\right)^{-1}\\
 &=\pi(R_r)^{-1}.
\end{aligned}
\]
Set $C:=R_0L_0$.
Then $\pi(C)=\pi(R_0)\pi(L_0)=e$,
so $C\in\ker\pi=E_2$. We next show that this element $C$ belongs to $K$.

Apply Lemma~\ref{lem:cap} with $G_j=E_j$ and $\ell_i=L_i$.  Let
$U=\{i_1<\cdots<i_k\}$ be a nonempty subset of $\{1,\ldots,d-1\}$, and set $B_U=L_{i_k}\cdots L_{i_1},\ H_U=R_{i_1}\cdots R_{i_k}$.

Consider
\[
 w_U=(+i_k,\ldots,+i_1,-i_1,\ldots,-i_k)
\]
and
\[
 \widetilde w_U=(+i_k,\ldots,+i_1,-0,+0,-i_1,\ldots,-i_k).
\]
For $A\subseteq\Z/d\Z$, let $\mathbf1_A\in E_d$ be its indicator vector, i.e., 
\[
 (\mathbf1_A)_r= \begin{cases}
                  1, r\in A\\
                  0, r\notin A, 
                  \end{cases}   r\in \Z/d\Z.
\]
The state path of $w_U$ is
\[
 \mathbf1_{\{0\}}
 \xrightarrow{+i_k}\cdots\xrightarrow{+i_1}
 \mathbf1_{\{0\}\cup U}
 \xrightarrow{-i_1}\cdots\xrightarrow{-i_k}
 \mathbf1_{\{0\}}.
\]
The state path of $\widetilde w_U$ is
\[
 \mathbf1_{\{0\}}
 \xrightarrow{+i_k}\cdots\xrightarrow{+i_1}
 \mathbf1_{\{0\}\cup U} \xrightarrow{-0}
 \mathbf1_U\xrightarrow{+0}
 \mathbf1_{\{0\}\cup U} 
 \xrightarrow{-i_1}\cdots\xrightarrow{-i_k}
 \mathbf1_{\{0\}}.
\]
Since $U\neq\varnothing$, we have $\mathbf1_U\neq0_d$.  Thus both
paths avoid $0_d$ except at their initial and terminal states, and hence $w_U,\widetilde w_U\in\cW$.
Lemma~\ref{lem:word-limit} implies  
\[
 B_UH_U\in K, \ B_UCH_U\in K.
\]
Since $K$ is a group,
\[
 (B_UCH_U)(B_UH_U)^{-1}=B_UCB_U^{-1}\in K.
\]
This holds for every nonempty $U\subseteq\{1,\ldots,d-1\}$.  
Lemma~\ref{lem:cap} therefore gives $C\in K$.  Since $K$ is a group, we also have $C^{-1}\in K$.

Choose $q\in\delta\cL$ with $\Phi(q)=C^{-1}$.
For every $N\in\N$, $\cL_{>N}\in q$.
Let $a,b\in D_0$ and $s\in\cL$ satisfy
\[
 a<\min\supp(s)\leq\max\supp(s)<b.
\]
Write  
\[
 s=((t_1,\varepsilon_1),\ldots,(t_k,\varepsilon_k))
\]
and set
\[
 \widetilde s
 =((a,+0),(t_1,\varepsilon_1),\ldots,(t_k,\varepsilon_k),(b,-0)).
\]
The state path of $\widetilde s$ is
\[
 0_d \xrightarrow{+0} e_0 \xrightarrow{s} e_0\xrightarrow{-0} 0_d.
\]
The middle path corresponding to $s$ is the realization of an admissible word, so it does not visit $0_d$.
By Lemma~\ref{lem:path-realisation}, there exists $ \alpha(a,s,b) \in \Sd$  such that 
\[
 \operatorname{Ev}( \alpha(a,s,b))=\widetilde s,
\]
and
\[
 \begin{aligned}
 T^{p_{ \alpha(a,s,b) }}
 &=v_a(+0)\Phi(s)v_b(-0)\\
 &=Q_a^{-1}\Phi(s)Q_b.
\end{aligned}   
\]

For fixed $a$ and $s$, the element $Q_a^{-1}\Phi(s)$ is a power of $T$ and  hence belongs to $\Lambda(E)$.  Since
\[
 \{b\in D_0:b>\max\supp(s)\}\in p_0,
\]
we have
\[
\begin{aligned}
 &\ulim{  p_0}
      {\substack{b\in D_0\\b>\max\supp(s)}}
 Q_a^{-1}\Phi(s)Q_b\\
 & =Q_a^{-1}\Phi(s)
   \left(\ulim{  p_0}{b\in D_0}Q_b\right)\\
 &=Q_a^{-1}\Phi(s)R_0.
\end{aligned}
\]
For fixed $a$,
\[
 \{s\in\cL:\min\supp(s)>a\}\in q.
\]
Left multiplication by $Q_a^{-1}$ and right multiplication by $R_0$ are continuous.  Hence
\[
\begin{aligned}
 &\ulim{q}{s:\,\min\supp(s)>a}
 Q_a^{-1}\Phi(s)R_0\\
 & =Q_a^{-1}
   \left(\ulim{q}{s\in\cL}\Phi(s)\right)R_0\\
 & =Q_a^{-1}C^{-1}R_0.
\end{aligned}
\]
Right multiplication by $C^{-1}R_0$ is continuous, so
\[
 \begin{aligned}
 &\ulim{  p_0}{a\in D_0}
  \ulim{q}{s:\,\min\supp(s)>a}
  \ulim{ p_0}
       {\substack{b\in D_0\\b>\max\supp(s)}}
 Q_a^{-1}\Phi(s)Q_b\\
  = & \left(\ulim{ p_0}{a\in D_0}Q_a^{-1}\right)C^{-1}R_0 
  = L_0C^{-1}R_0 \\
  =&L_0(R_0L_0)^{-1}R_0 =e.
\end{aligned}   
\]
 
Set $\mathcal R=\{T^{p_\alpha}:\alpha\in\Sd\}$, then $e\in\overline{\mathcal R}$.
For $x\in X$ and a neighbourhood $U$ of $x$, set
\[
 \mathcal O_U=\{g\in E:gx\in U\}.
\]
The evaluation map $g\mapsto gx$ is continuous and $e\in\mathcal O_U$, so
$\mathcal O_U$ is a neighbourhood of $e$.  Hence $\mathcal O_U\cap\mathcal R\neq\varnothing$.
Choose $\alpha\in\Sd$ such that
$T^{p_\alpha}\in\mathcal O_U$.  Then $T^{p_\alpha}x\in U$,
and hence $p_\alpha\in N(x,U)$.

It remains to treat the case $d=1$, where no nontrivial commutator
cancellation is needed.
Now let $d=1$.  Then $E_2=\{e\}$, so $E$ is a compact abelian topological
group.  Continuity of inversion implies
\[
 L_0=\ulim{  p_0}{a\in D_0}Q_a^{-1}
 =\left(\ulim{  p_0}{a\in D_0}Q_a\right)^{-1}=R_0^{-1}.
\]
For $a<b$ in $D_0=\N$, set
\[
 \alpha=[a,b-1]\cap\N.
\]
Then $\alpha\in\mathcal S_1$ and
\[
 T^{p_\alpha}=Q_a^{-1}Q_b.
\]
Since $\{b\in D_0:b>a\}\in p_0$,
\[
 \ulim{ p_0}{a\in D_0}
 \ulim{  p_0}{\substack{b\in D_0\\b>a}}Q_a^{-1}Q_b
  =L_0R_0 =e.
\]
Therefore
\[
 e\in
 \overline{\{T^{p_\alpha}:\alpha\in\mathcal S_1\}}.
\]
 The same evaluation argument as above now completes the proof.
\end{proof}

We finally apply the recurrence theorem to nilsystems.

\begin{proof}[Proof of Theorem~\ref{thm:nilbohr}]
Let $A$ be a $\NilBohr{d}$ set.  Choose a $d$-step nilsystem $G/\Gamma$, an
element $g\in G$, and a neighbourhood $U$ of $\Gamma$ such that
\[
 B:=\{n\in\Nzero:g^n\Gamma\in U\}\subseteq A.
\]
Set
\[
 Y=\overline{\{g^n\Gamma:n\in\Z\}}.
\]
Then $(Y,T_g)$ is a minimal $d$-step nilsystem.
By Lemma~\ref{pro-nilsystem-top-nilpotent group}, the Ellis group $E$ of $(Y,T_g)$ is $d$-step top-nilpotent.  Thus its topological lower central series
\[
E_1^{\mathrm{top}}=E,\ E_{j+1}^{\mathrm{top}}=\overline{[E_j^{\mathrm{top}},E]}
\]
satisfies $E_{d+1}^{\mathrm{top}}=\{e\}$.
By the definition of $E_j^{\mathrm{top}}$ and Lemma~\ref{E_d^top subgroup},
each $E_j^{\mathrm{top}}$ is a closed subgroup of $E$.  Hence
\[
E_1^{\mathrm{top}}\supseteq E_2^{\mathrm{top}}\supseteq\cdots
\supseteq E_{d+1}^{\mathrm{top}}=\{e\}
\]
is a filtration of the type required in Theorem~\ref{thm:dynamical}.
Therefor, $B\in\SG_d^*$.
Since $B\subseteq A$, $A\in\SG_d^*$.

\end{proof}

\noindent \textbf{Acknowledgment.}
J. Zhao is supported by NSF of China (Grant nos.~12301226 and 12371190).

\end{document}